%% file: Rigidite07.tex
\documentclass{amsart}
\usepackage{amsmath,ifthen,amsthm,srcltx,amsopn,amssymb,amsfonts}
\usepackage{amscd,amssymb,bbm}
\usepackage{asymptote}
\usepackage[T1]{fontenc}
\usepackage{inputenc}
\usepackage{psfrag}
\usepackage{delarray,graphicx}
\usepackage{comment}
\usepackage[small,nohug]{diagrams}
\usepackage[all,cmtip]{xy}

\usepackage{latexsym,amssymb,amsmath}
\RequirePackage{graphics,epsfig,psfrag}
\usepackage{amsfonts}
\usepackage{epsfig,graphics}
\usepackage{graphicx}

\usepackage{amssymb,latexsym}
\usepackage{color}

\usepackage{amsthm}
\usepackage{eucal}

\newcommand{\showcomments}{yes}
\renewcommand{\showcomments}{no}

\newsavebox{\commentbox}
{\ifthenelse{\equal{\showcomments}{yes}}%
{\footnotemark
        \begin{lrbox}{\commentbox}
        \begin{minipage}[t]{1.25in}\raggedright\sffamily\tiny
        \footnotemark[\arabic{footnote}]}
{\begin{lrbox}{\commentbox}}}%
{\ifthenelse{\equal{\showcomments}{yes}}%
{\end{minipage}\end{lrbox}\marginpar{\usebox{\commentbox}}}
{\end{lrbox}}}

\title[Local rigidity in $\PGL(3, \bC)$]{Local rigidity for $\PGL(3 , \bC)$-representations of $3$-manifold groups}

\author[Bergeron, Falbel, Guilloux, Koseleff, Rouillier]{Nicolas Bergeron, Elisha Falbel, Antonin Guilloux,\\ Pierre-Vincent
Koseleff and Fabrice Rouillier} 
\thanks{N.B. is a member of the Institut Universitaire de France.}
\address{Institut de Math\'ematiques de Jussieu \\
Unit\'e Mixte de Recherche 7586 du CNRS \\
Universit\'e Pierre et Marie Curie \\
4, place Jussieu 75252 Paris Cedex 05, France \\}
\email{bergeron@math.jussieu.fr \\ falbel@math.jussieu.fr \\
aguillou@math.jussieu.fr \\ koseleff@math.jussieu.fr \\
fabrice.rouillier@inria.fr}

\urladdr{http://people.math.jussieu.fr/~bergeron \\
http://people.math.jussieu.fr/~falbel \\
http://people.math.jussieu.fr/~aguillou\\
http://people.math.jussieu.fr/~koseleff\\
}

\DeclareFontFamily{OT1}{rsfs}{}

\DeclareFontShape{OT1}{rsfs}{n}{it}{<-> rsfs10}{}
\DeclareMathAlphabet{\mathscr}{OT1}{rsfs}{n}{it}

\newcommand{\Z}{\mathbb{Z}}

\newcommand{\m}{\mathfrak{m}}

\DeclareFontFamily{OT1}{rsfs}{}

\DeclareFontShape{OT1}{rsfs}{n}{it}{<-> rsfs10}{}
\DeclareMathAlphabet{\mathscr}{OT1}{rsfs}{n}{it}
\newcommand{\Q}{\mathbb{Q}}

\newcommand{\R}{\mathbb{R}}

\swapnumbers
\newtheorem{theorem}[subsection]{Theorem}
\newtheorem{lemma}[subsection]{Lemma}
\newtheorem{proposition}[subsection]{Proposition}

\newtheorem*{theorem*}{Theorem}
\newtheorem*{appl*}{Application}

\theoremstyle{definition}

\theoremstyle{remark}
\newtheorem*{remark}{Remark}

\newcommand\bZ{\mathbb{Z}}
\newcommand\bC{\mathbb{C}}
\newcommand\bR{\mathbb{R}}

\newcommand\SL{\textrm{SL}}
\newcommand\PGL{\textrm{PGL}}
\newcommand\PU{\textrm{PU}}

\numberwithin{equation}{subsection}

\newcommand{\cal}{\mathcal}

\newcommand{\PSL}{\mathrm{PSL}}

\def\adots{\mathinner{\mkern2mu\raise1pt\hbox{.}
\mkern3mu\raise4pt\hbox{.}\mkern1mu\raise7pt\hbox{.}}}

\begin{document}

\begin{abstract}  
Let $M$ be a non-compact hyperbolic $3$-manifold that has a
triangulation by positively oriented ideal tetraedra. We explain how
 to produce local coordinates for the
variety defined by the gluing equations for $\PGL (3 ,
\bC)$-representations. In particular we prove local rigidity of the
``geometric'' representation in $\PGL (3 , \bC)$, recovering a recent
result of Menal-Ferrer and Porti.  
More generally we give a criterion for local rigidty of $\PGL (3 ,
\bC)$-representations and provide detailed analysis of the figure
eight knot sister manifold exhibiting the different possibilities
that can occur. 
\end{abstract}

\maketitle

\tableofcontents

\section{Introduction}

Let $M$ be a compact orientable $3$-manifold with boundary a union of
$\ell$ tori. Assume that the interior of $M$ carries a hyperbolic
metric of finite volume and let $\rho_{} : \pi_1 (M) \rightarrow
\mathrm{PGL}(3, \bC)$ be the corresponding holonomy composed with the
$3$-dimensional irreducible representation of $\mathrm{PGL} (2 , \bC)$ (this representation is usually called geometric
or adjoint representation).  

Building on \cite{BFG} we give a combinatorial proof of the following
theorem first proved by Menal-Ferrer and Porti
\cite{MenalFerrerPorti}.  

\begin{theorem} \label{T1}
The class $[\rho_{}]$ of $\rho$ in the algebraic quotient of
$\mathrm{Hom} (\pi_1 (M) , \mathrm{PGL}(3 , \bC))$ by the 
action of $\mathrm{PGL} (3 , \bC)$ by conjugation is a smooth point with local dimension $2\ell$.
\end{theorem}

Our main theorem \ref{theo:variety}
 is in fact more general.  We do not solely consider the geometric representation
and in fact our proof applies to an 
explicit open subset (called  $\mathcal{R}(M , \mathcal{T}^+)$, see subsection \ref{plus}) of the (decorated) representation variety into $\PGL(3,\bC)$. It also provides
explicit coordinates and a description of the possible deformations. We analyse
in the last section the figure-eight knot sister manifold: we describe all the
(decorated) representations whose restriction to the boundary torus are
unipotent. It turns out that there exist rigid points (i.e. isolated points in
the (decorated) unipotent representation variety) together with non-rigid
components. 

There is a natural holonomy map (see section \ref{symplectic}) from the (decorated) representation variety of $M$ to
the representation variety of its boundary. It is known that its image is a
Lagrangian subvariety and the map is a local isomorphism on a Zariski-open set. Our remark in subsection 
\ref{lagrangian} proves in a combinatorial way these facts. When $M$ is a knot
complement and one considers the group $\mathrm{PGL}(2,\bC)$ instead of
$\mathrm{PGL}(3,\bC)$, this image is the algebraic variety defined by the
$A$-polynomial of the knot. In this paper, we explore more precisely the
map $\mathrm{hol}$ and exhibit a fiber which is not discrete. 

This research was in part financed by the ANR project {\sl Structures G\'eom\'etriques Triangul\'ees}.


\section{Ideal triangulation}

\subsection{} An {\it ordered simplex} is a simplex with a fixed
vertex ordering. Recall that an orientation of a set of vertices is a
numbering of the elements of this set up to even permutation.  
The face of an ordered simplex inherits  an orientation.  
We call {\it abstract triangulation} a pair $\mathcal{T}=( (T_{\mu}
)_{\mu = 1 , \ldots , \nu} , \Phi)$ where $(T_{\mu} )_{\mu = 1 ,
  \ldots , \nu}$ is a finite family of abstract ordered simplicial
tetrahedra and $\Phi$ is a matching of the faces of the $T_{\mu}$'s
reversing the orientation.  
For any simplicial tetrahedron $T$, we define $\mathrm{Trunc} (T)$ as
the tetrahedron truncated at each vertex. The space obtained from $\mathrm{Trunc} (T_{\mu})$  after matching the faces
will be denoted by $K$.  

We call {\it triangulation}
--- or rather {\it ideal triangulation} --- of a compact $3$-manifold
$M$ with boundary an abstract triangulation $\mathcal{T}$ and an oriented homeomorphism  
$$ K= \bigsqcup_{\mu =1}^{\nu} \mathrm{Trunc} (T_{\mu}) / \Phi \rightarrow M .$$

In the following we will always assume that the boundary of $M$ is a
disjoint union of a finite collection of $2$-dimensional tori. 
Recall that, by a simple Euler characteristic count, the number of edges of $K$ is equal to the number $\nu$ of tetrahedra. The
most important family of examples being the compact $3$-manifolds
whose interior carries a complete hyperbolic structure of finite
volume. The existence of an ideal triangulation for $M$ still appears
to be an open question.\footnote{Note however that starting from the
  Epstein-Penner decomposition of $M$ into ideal polyhedra, Petronio
  and Porti \cite{PetronioPorti} produce a degenerate triangulation of
  $M$.} Luo, Schleimer and Tillmann \cite{LST} nevertheless prove
that, passing to a finite regular cover, we may assume that $M$ admits
an ideal triangulation.  
In the following paragraphs we assume that $M$ itself admits an ideal
triangulation $\mathcal{T}$ and postpone to the proof of Theorem
\ref{T1} the task of reducing to this case (see lemma \ref{lem:fincov}).

\subsection{Parabolic decorations} 
We recall from \cite{BFG} the
notion of a {\it parabolic decoration} of the pair $(M,\mathcal{T})$:
to each tetrahedron $T_{\mu}$ of $\mathcal{T}$ we associate non-zero
complex coordinates $z_{\alpha} (T_{\mu})$ ($\alpha \in I$) where  
$$I = \{ \mbox{vertices of the (red) arrows in the triangulation given by Figure \ref{FG}} \}.$$
\begin{figure}[ht]      
\begin{center}
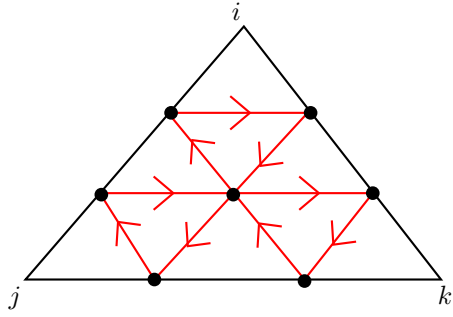
\caption{Combinatorics of $W$} \label{FG}
\end{center}
\end{figure}
Let $J_{T_{\mu}}^2 = \mathbb{Z}^I$ be the $16$-dimensional abstract
free $\mathbb{Z}$-module and denote the canonical basis $\{e_{\alpha}
\}_{\alpha \in I}$ of $J_{T_{\mu}}^2$.  
It contains {\it oriented
  edges} $e_{ij}$ (edges are oriented from $j$ to $i$) and {\it faces}
$e_{ijk}$. Using these notations the $16$-tuple of complex parameters
$(z_{\alpha} (T_{\mu}))_{\alpha \in I}$ is better viewed as an element    
$$z(T_{\mu})  \in \mathrm{Hom}(J^2_{T_{\mu}} , \bC^{\times})
\cong \bC^{\times} \otimes_{\Z} (J_{T_{\mu}}^2 )^*.$$ 
We refer to \cite{BFG} for details. 
Such an element uniquely determines a tetrahedron of flags if and only
if the following relations are satisfied: 
\begin{equation} \label{Rel1}
z_{ijk}= - z_{il}z_{jl}z_{kl},
\end{equation}
and
\begin{equation} \label{Rel3}
z_{ik} =  \frac{1}{1-z_{ij}}.
\end{equation}

Remark that the second relation implies the following one:
\begin{equation} \label{Rel2}
z_{ij} z_{ik} z_{il} = -1, 
\end{equation}

\begin{figure}[ht]      
\begin{center}
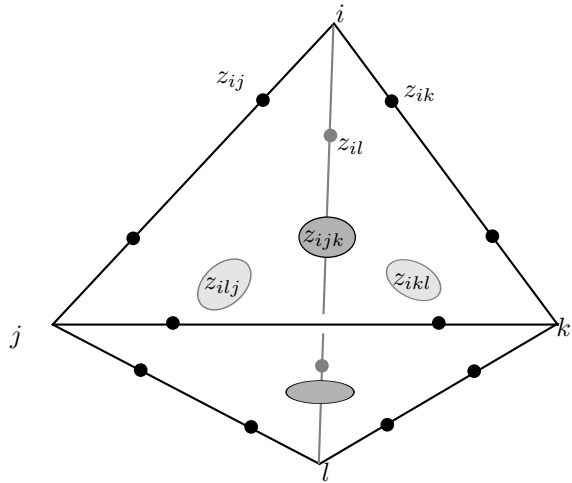
\caption{The $z$-coordinates for a tetrahedron} \label{tetra2}
\end{center}
\end{figure}

\subsection{} Let $J^2$ denote the direct sum of the
$J_{T_{\mu}}^2$'s and consider an element $z \in \bC^{\times}
\otimes_{\Z} (J^2)^*$  as a set of parameters of the triangulation
$\mathcal{T}$. 
As usual, these coordinates are subject to consistency relations after
gluing by $\Phi$: given two adjacent tetrahedra $T_{\mu}$, $T_{\mu'}$
of $T$ with a common face  
$(ijk)$ then 
\begin{equation} \label{face}
z_{ijk} (T_{\mu}) z_{ikj} (T_{\mu'}) = 1.
\end{equation}
And given a sequence $T_1, \ldots , T_{\mu}$ of tetrahedra sharing a
common edge $ij$ and such that $ij$ is an inner edge of the sub
complex composed by $T_1\cup \cdots \cup T_{\mu}$ 
then 
\begin{equation} \label{edge}
z_{ij} (T_1) \cdots z_{ij} (T_{\mu}) = z_{ji} (T_1) \cdots z_{ji} (T_{\mu}) = 1.
\end{equation}

\subsection{} \label{rem:deco} Consider a fundamental domain of the
triangulation of the universal cover $\tilde M$ lifted from the one of $M$. A
decoration of the complex is then equivalent to
an assignment of a flag to each of its vertices; together with
an additional transversality condition on the flags to ensure that the
$z_\alpha$'s do not vanish.

\section{The representation variety} 

Given $M$ and a triangulation $\mathcal{T}$ we consider the space of parabolic decorations and denote it by  $\mathcal{R}(M,\mathcal{T})$ (we call it {\it representation variety} associated to parabolic decorations of a triangulation).  It is observed in the next subsection that it can be identified to an open subset of $\mathrm{Hom}(\pi_1(M),\PGL(3,\bC))/\PGL(3,\bC)$.

More explicitly we  define the 
 $\mathcal{R}(M,\mathcal{T})$ as:
$$\mathcal{R} (M,\mathcal{T}) = g^{-1} (1 , \ldots , 1)$$
where $g=(h,a,f) : \bC^{\times} \otimes (J^2)^* \to (\bC^{\times})^{8
  \nu} \times (\bC^{\times})^{4 \nu} \times  (\bC^{\times})^{4 \nu} \cong {(\bC^\times)}^{16\nu}$ is the product of the three
maps $h$, $a$, $f$, defined below.

\subsection{} First $h = (h_1 , \ldots , h_{\nu})$ is the product of
the maps $h_{\mu} : \bC^{\times} \otimes_{\Z} (J_{T_{\mu}}^2)^*
\rightarrow {(\bC^\times)}^8$ ($\mu = 1 , \ldots ,\nu$) associated to the   
$T_{\mu}$'s and which are defined by
\begin{multline*}
h_{\mu} (z) = \Big( - \frac{z_{ijk}}{z_{il} z_{jl} z_{kl}} , -
\frac{z_{ikl}}{z_{ij} z_{kj} z_{lj}} , - \frac{z_{ilj}}{z_{ik} z_{lk}
  z_{jk}}, - \frac{z_{kjl}}{z_{ki} z_{ji} z_{lj}} , \\ -z_{ij}z_{ik}
z_{il} , -z_{ji}z_{jk} z_{jl} , -z_{ki}z_{kj} z_{kl} , -z_{li}z_{lj}
z_{lk}    \Big) 
\end{multline*}
here $z = z(T_{\mu}) \in \bC^{\times} \otimes_{\Z} (J_{T_{\mu}}^2)^*$,
cf. \eqref{Rel1} and \eqref{Rel2}.  

\subsection{} Next we define the map $a$, cf. \eqref{Rel3}. Let 
$a_{\mu} : \bC^{\times} \otimes_{\Z}
(J_{T_{\mu}}^2)^* \rightarrow \bC^4$ ($\mu = 1 , \ldots \nu$)
associated to $T_{\mu}$
be the map defined by 
$$a_{\mu} (z) = (z_{ik} (1- z_{ij}) , z_{jl} (1-z_{ji}) , z_{ki} (1-z_{kl}) , z_{lj} (1-z_{lk})).$$
We define then 
$a= (a_1 , \ldots , a_{\nu})$.

\subsection{} Finally we let $C_1^{\rm or}$ be the free $\Z$-module
generated by the oriented $1$-simplices of $K$ and $C_2$ the
free $\Z$-module generated by the $2$-faces of $K$.  As observed before, in the case $K$
has only tori as ideal boundaries, the number of edges in $K$ is $\nu$ and the number of faces is $2\nu$. Therefore
the  $\Z$-module $C^{\rm or}_1+ C_2$ has rank $4\nu$ and therefore  $\mathrm{Hom} (C_1^{\rm or} + C_2 , \bC^{\times})\cong (\bC^{\times})^{4 \nu}$.

 As in
 \cite{BFG}, we define a map  
$$
F : C^{\rm or}_1+ C_2 \rightarrow  J^2
$$
by, for $\bar e_{ij}$ an oriented edge of $K$,  
$$
F(\bar e_{ij})= e_{ij}^{1} + \ldots + e_{ij}^{\mu} 
$$
where $T_{1}, \ldots , T_{\mu}$ is a sequence of tetrahedra sharing the edge $\bar e_{ij}$
such that $\bar e_{ij}$ is an inner edge of the subcomplex $T_{1} \cup
\cdots \cup T_{\mu}$ and each $e_{ij}^{\mu}$ gets identified with the
{\it oriented} edge $\bar e_{ij}$ in $\mathcal{T}$. And for a $2$-face
$\bar e_{ijk}$,  
$$
F(\bar e_{ijk})=e_{ijk}^{\mu} + e_{ikj}^{\mu'},
$$
where $\mu$ and $\mu'$ index the two 
$3$-simplices having the common face $\bar e_{ijk}$.
We then define the map 
$$
f: \mathrm{Hom} (J^2 , \bC^{\times}) \rightarrow \mathrm{Hom} (C_1^{\rm or} + C_2 , \bC^{\times})\cong (\bC^{\times})^{4 \nu} 
$$
by $f(z) = z \circ F$, compare \eqref{face} and \eqref{edge}.  A decoration $z \in \bC^{\times}
\otimes_{\Z} (J^2)^*$  satisfies the edge and face equations (\ref{edge} and \ref{face}) if and only if $f(z)=1$ (compare with the map $F^*$ defined in the next section 
so we can write equivalently $z\in \bC^\times\otimes_\bZ \mathrm{Ker}(F^*)$).

\subsection{} 
\label{representation}
From an element in $\mathcal R(M,\mathcal T)$, one may
reconstruct a representation (up to conjugacy) by computing the holonomy of the
complex of flags (see \cite[section 5]{BFG}). Restating Remark
\ref{rem:deco}, a decoration is equivalent to a map, equivariant under
$\pi_1(M)$, from the space of cusps of $\tilde M$ to the space of flags with a
transversality condition. Note that each flag is then invariant by the holonomy
of the cusp.

Moreover, the map from $\mathcal
R(M,\mathcal T)$ to $\mathrm{Hom}(\pi_1(M),\PGL(3,\bC))/\PGL(3,\bC)$ is open:
given  a representation $\rho$, its decoration equip each cusp $p$ of $M$
with a flag $F_p$ invariant by the holonomy of the isotropy $\Gamma_p$ of $p$.
Now,  deforming the representation $\rho$ to $\rho'$, for each cusp $p$, one 
can deform $F_p$ into a flag $F'_p$ invariant under $\rho'(\Gamma_p)$. The
transversality condition being open, this gives a decoration for any decoration
$\rho'$ near $\rho$.  

Generalizations of this formalism to representations of 3-dimensional fundamental groups to $\PGL(n,\bC)$ for $n\geq 3$ can be seen in
\cite{GGZ,DGG}.

\section{The symplectic isomorphism}\label{symplectic}

In this section we recall results of \cite{BFG} which will be used in the proof of the main theorem.
As in \cite{BFG}  each
$J_{T_{\mu}}^2$ is equipped with  
a bilinear skew-symmetric form given by 
$$\Omega^2 ( e_{\alpha} , e_{\beta} ) = \varepsilon_{\alpha \beta}.$$
Here given $\alpha$ and $\beta$ in $I$ we set (recall figure \ref{FG}):
$$\varepsilon_{\alpha \beta} = \# \{ \mbox{oriented (red) arrows from
  $\alpha$ to $\beta$}\} -  \# \{ \mbox{oriented (red) arrows from
  $\beta$ to $\alpha$}\}.$$ 

We let $(J^2 , \Omega^2)$ denote the orthogonal sum of the spaces
$(J^2_{T_{\mu}} , \Omega^2)$. We denote by $e_{\alpha}^{\mu}$ the
$e_{\alpha}$-element in $J_{T_{\mu}}^2$. 
Let 
$$p : J^2 \rightarrow (J^2)^*$$
be the homomorphism $v \mapsto \Omega^2 (v , \cdot )$. On the basis
$(e_{\alpha})$ and its dual $(e_{\alpha}^*)$, we can write 
$$p (e_{\alpha}) = \sum_{\beta} \varepsilon_{\alpha \beta} e_{\beta}^*.$$
Let $J$ be the quotient of $J^2$ by the kernel of $\Omega^2$. The
latter is the subspace generated on each tetrahedron by elements of
the form  
$$\sum_{\alpha \in I} b_{\alpha} e_{\alpha}$$
for all $\{ b_{\alpha} \} \in \mathbb{Z}^I$ such that $\sum_{\alpha
  \in I} b_{\alpha} \varepsilon_{\alpha \beta} =0$ for every $\beta
\in I$. 
Equivalently it is the subspace generated by $e_{ij}+e_{ik}+e_{il}$ and $e_{ijk}- (e_{il}+e_{jl}+e_{kl})$.
 
We let $J^*  \subset (J^2)^*$ be the dual subspace which consists of
the linear maps which vanish on the kernel of $\Omega^2$. Note that  
we have $J^* = \mathrm{Im} ( p)$ and that it is $8$-dimensional. The
form $\Omega^2$ induces a  non-degenerate  skew-symmetric (we will call it symplectic) form
$\Omega$ on $J$. 
This yields a canonical identification between $J$ and $J^*$; we
denote by $\Omega^*$ the corresponding symplectic form on $J^*$. 

Consider the sequence introduced in \cite{BFG}:
$$
{C_1^{\rm or} + C_2 \stackrel{F}{\rightarrow} J^2
\stackrel{p}{\rightarrow} (J^2)^* \stackrel{F^*}{\rightarrow} C_1^{\rm
  or} + C_2}. 
 $$ 
The skew-symmetric form $\Omega^*$ on $J^*$ is non-degenerate but its restriction to 
$\mathrm{Im} ( p)\cap \mathrm{Ker} (F^*)$ has a kernel. In \cite{BFG}
we relate this form with ``Goldman-Weil-Petersson'' forms on the
peripheral tori: there is a form $\textrm{wp}_s$ on each
$H^1(T_s,\bZ^2)$, $s=1, \ldots , \ell$, defined as the coupling of the
cup product on $H^1$ with the scalar product $\langle , \rangle$ on
$\bZ^2$ defined by: 
\footnote{This product should be interpreted as the Killing form on
  the space of roots of $\mathfrak{sl}(3,\bC)$ through a suitable
  choice of basis.} 
$$\langle
\begin{pmatrix}n\\m\end{pmatrix},
\begin{pmatrix}n'\\m'\end{pmatrix}\rangle
=\frac{1}{3}(2nn'+2mm'+nm'+n'm),$$ 
see \cite[section 7.2]{BFG}. 

For our purpose we rephrase the content of \cite[Corollary 7.11]{BFG} in the following:

\begin{proposition} \label{L4} 
We have $\mathrm{Ker} (\Omega^*_{| \mathrm{Im} ( p) \cap \mathrm{Ker}
  (F^*)} ) = \mathrm{Im} (p \circ F)$. The skew-symmetric form
$\Omega^*$ therefore induces a {\rm symplectic} form on the quotient 
$$(J^* \cap \mathrm{Ker} (F^*)) / \mathrm{Im} (p\circ F).$$
Moreover: there is a symplectic isomorphism --- defined over $\Q$ --- between
this quotient and the space $\oplus_{s=1}^{\ell} H^1 (T_s,\bZ^2)$ equipped with
the direct sum $\oplus_s \mathrm{wp}_s$, still denoted $\mathrm{wp}$.
\end{proposition}

\subsection{} Let us briefly explain how to understand Proposition \ref{L4} as a corollary in \cite{BFG}. First recall from 
\cite[section 7.3]{BFG} that given an element $z \in \mathcal{R}(M ,
\mathcal{T})$ we may compute the holonomy of a loop 
$c \in H_1 (T_s,\bZ)$ and get an upper triangular matrix; let
$(\frac{1}{C^*} , 1 , C)$ be its diagonal part. The application which
maps $c \otimes \begin{pmatrix}n\\m\end{pmatrix}$ to $C^m (C^*)^n$ yields
 the {\it holonomy map}
$$\label{holonomy}
\mathrm{hol} : \mathcal{R}(M , \mathcal{T}) \to \oplus_{s=1}^{\ell}
\mathrm{Hom} (H_1 (T_s, \bZ^2) , \bC^{\times}).
$$ 
The symplectic map of the proposition is the linearization of this holonomy map. Here is how it is done:
our variety $\mathcal{R}(M , \mathcal{T})$ is a subvariety of $\bC^\times \otimes (J^2)^*$. This last space may be viewed as the exponential of the $\bC$-vector space $\bC\otimes (J^2)^*$. Lemma 7.5 of \cite{BFG}\footnote{there, the holonomy map is denoted $R$.} expresses the square of $\mathrm{hol}$ as the exponential of a linear map:
\begin{equation} \label{dhol}
\bC \otimes (J^2)^* \to \oplus_{s=1}^{\ell}H^1(T_s,\bC^2)\simeq \oplus_{s=1}^{\ell}\mathrm{Hom} (H_1 (T_s, \bZ^2) , \bC).
\end{equation}
Moreover this map is defined over $\Q$ and at the level of the $\Z$-modules. At this level, it is indeed obtained as the composition of the map $h^*$, dual to the map $h$ defined in \cite[section 7.4]{BFG}, with the projection 
to $\oplus_s H^1 (T_s , \bZ^2) \cong  \bZ^{4\ell}$ (using a symplectic basis of  $H_1 (T_s , \bZ)$).
 The symplectic isomorphism of Proposition \ref{L4} is given by this map \cite[Theorem 7.9 and Corollary 7.11]{BFG}, after restriction to $J^* \cap \mathrm{Ker} (F^*)$ and quotienting by $\mathrm{Im}(p\circ F)$ (see \cite[section 7.6]{BFG}).

\section{Infinitesimal deformations}

\subsection{} Let $ z= (z (T_{\mu}))_{\mu = 1 , \ldots , \nu} \in
\mathcal{R}(M,\mathcal{T})$. The exponential map identifies $T_z (
\bC^{\times} \otimes_{\Z} (J^2)^*)$ with  
$\bC \otimes (J^2)^* = \mathrm{Hom} (J^2 , \bC)$. Under this
identification the differential $d_z g$ defines a linear map which we
write as a direct sum $d_z h \oplus d_z a \oplus d_z f$. 

In the following three lemmas we identify the kernel of each of these
three linear maps in order to prove Proposition \ref{kernel}.

\begin{lemma} \label{L1}
As a subspace of $\bC \otimes (J^2)^*$ the kernel of $d_z h$ is equal to $\bC \otimes J^*$.
\end{lemma}
\begin{proof} It follows from the definitions that $\xi \in \bC
  \otimes (J^2)^*$ belongs to the kernel of $d_z h$ if and only if it
  vanishes on the subspace $\mathrm{Ker} (\Omega^2 )$ generated by
  $e_{ij}^{\nu}+e_{ik}^{\nu}+e_{il}^{\nu}$ and $e_{ijk}^{\nu}-
  (e_{il}^{\nu}+e_{jl}^{\nu}+e_{kl}^{\nu})$. 
This concludes the proof.
\end{proof}

\begin{lemma} \label{L2}
As a subspace of $\bC \otimes (J^2)^*$ the kernel of $d_z a$ is equal
to the subspace $\mathcal{A} (z)$ defined as:
$$\left\{ \xi \in \mathrm{Hom} (J^2 , \bC) \; : \; \begin{array}{l}
    \xi (e^{\mu}_{ij}) + z_{il} (T_{\mu}) \xi (e^{\mu}_{ik})=0, \ \xi
    (e^{\mu}_{ji}) + z_{jk} (T_{\mu}) \xi (e^{\mu}_{jl})=0 \\ \xi
    (e^{\mu}_{ki}) + z_{kl} (T_{\mu}) \xi (e^{\mu}_{kj})=0, \ \xi
    (e^{\mu}_{lj}) + z_{lk} (T_{\mu}) \xi (e^{\mu}_{li})=0
  \end{array}, \forall  \mu \right\}.$$ 
\end{lemma}
\begin{proof} Here again we only have to check this on each tetrahedra
  $T_{\mu}$ of $\mathcal{T}$.  
All four coordinates of $a_{\mu}$ can be dealt with in the same way,
we only consider the first coordinate:  
$$z \mapsto z_{ik} (1-z_{ij}).$$
Taking the differential of the logarithm we get:
$$\frac{dz_{ik}}{z_{ik}} - \frac{dz_{ij}}{1-z_{ij}} =0.$$
Equivalently,
$$\frac{dz_{ij}}{z_{ij}} = \left( \frac{1-z_{ij}}{z_{ij}} \right) \frac{dz_{ik}}{z_{ik}}.$$
Since $z  \in \mathcal{R}(M,\mathcal{T})$, we have $h_{\nu} (z) = a_{\mu} (z) =1$. In particular 
$$(1-z_{ij}) = \frac{1}{z_{ik}} \quad \mbox{and} \quad z_{ij} z_{ik} = - \frac{1}{z_{il}}.$$
We conclude that 
$$\frac{dz_{ij}}{z_{ij}} + z_{il} \frac{dz_{ik}}{z_{ik}}=0.$$
Under the identification of $T_z ( \bC^{\times} \otimes_{\Z} (J^2)^*)$ with 
$\bC \otimes (J^2)^* = \mathrm{Hom} (J^2 , \bC)$ this proves the lemma.
\end{proof}

We denote by $F^* : (J^2)^* \rightarrow C_1^{\rm or} + C_2$ the dual
map to $F$ (here we identify $C_1^{\rm or} + C_2$ with its dual by
using the canonical basis). It is the ``projection map'': 
$$(e_{\alpha}^{\mu})^* \mapsto \bar e_{\alpha}$$
when $(e_{\alpha}^{\mu})^* \in (J^2)^*$ 
. By definition of $f$ we have: 

\begin{lemma} \label{L3}
As a subspace of $\bC \otimes (J^2)^*$ the kernel of $d_z f$ is equal
to $\bC \otimes \mathrm{Ker} (F^*)$. 
\end{lemma}

Lemma \ref{L1}, \ref{L2} and \ref{L3} clearly imply the next proposition.
\begin{proposition}\label{kernel}
\begin{equation}
\mathrm{Ker} \ d_z g = ( \bC \otimes (\mathrm{Im} ( p ) \cap
\mathrm{Ker} (F^* ) ) \cap \mathcal{A}(z) . 
\end{equation}
\end{proposition}

Note that among these three spaces, two are defined over $\Z$ and do not depend
on the point $z$, but the last one, $\mathcal A(z)$, is actually depending on
$z$. We shall give examples where the dimension of the intersection vary
and describe the corresponding deformations in $\mathcal{R}(M ,
\mathcal{T})$. But first we consider an open subset 
of $\mathcal{R} (M , \mathcal{T})$ which we prove to be a manifold. 

\section{The complex manifold $\mathcal{R}(M , \mathcal{T}^+)$}

\subsection{}\label{plus} Let 
$$
\mathcal{R}(M , \mathcal{T}^+) = \left\{ z= (z (T_{\mu}))_{\mu = 1 ,
    \ldots , \nu} \in \mathcal{R}(M , \mathcal{T}) \; : \; \mathrm{Im}
  \ z_{ij} (T_{\mu}) >0 , \ \forall \mu, i,j\right\}
  $$  
be the subspace of $\mathcal{R}(M , \mathcal{T})$ whose {\it edge}
coordinates have positive imaginary parts. Note that coordinates
corresponding to the geometric representation belong to $\mathcal{R}(M
, \mathcal{T}^+)$.  

\begin{remark} Observe that in the case of an ideal triangulation of a hyperbolic
manifold with  shape parameters having all positive imaginary part and satisfying the 
edge conditions and unipotent holonomy conditions we obtain as holonomy the
geometric representation
$\rho_{geom}$.  The shape parameters in the $\PSL(2,\bC)$ case give rise 
to a parabolic decoration of the ideal triangulation in the sense of this paper which is clearly contained in $\mathcal{R}(M , \mathcal{T}^+)$.  This is explained in detail
in \cite{BFG}. 
\end{remark}

The main theorem of this section is a generalization of a theorem of
Choi \cite{Choi}; it  states that $\mathcal{R}(M , \mathcal{T}^+)$ is
a smooth complex manifold and gives local coordinates. 

Recall  we assumed that $\partial M$ is the disjoint union of $\ell$ tori. For
each boundary torus $T_s$ ($s=1 , \ldots , \ell$) of $M$ we fix a
symplectic basis $(a_s , b_s)$ of  
the first homology group $H_1 (T_s)$. Given a point $z$ in the
representation variety $\mathcal{R}(M , \mathcal{T})$ we may consider
the holonomy elements associated to $a_s$, resp. $b_s$. They preserve
a flag associated to the torus by the decoration. In a basis adapted
to this flag, those matrices are of the form (for notational
simplicity, we write them in $\PGL(3,\bC)$ rather than $\SL(3,\bC)$): 
$$\begin{pmatrix} \frac{1}{A_s^*} & *
  &*\\0&1&*\\0&0&A_s\end{pmatrix}\textrm{ and }\begin{pmatrix}
  \frac{1}{B_s^*} & * &*\\0&1&*\\0&0&B_s\end{pmatrix}.$$  
Now the diagonal entries of
the first matrix $A_s$ and $A_s^*$ for each torus define a map 
\begin{equation} \label{hol}
\mathcal{R}(M , \mathcal{T}) \rightarrow (\bC^{\times})^{2\ell}; \quad
z \mapsto (A_s , A_s^*)_{s=1 , \ldots , \ell}.  
\end{equation}

\begin{theorem}\label{theo:variety}
Assume that $\partial M$ is the disjoint union of $\ell$ tori. Then
the complex variety $\mathcal{R}(M , \mathcal{T}^+)$ is a smooth
complex manifold of dimension $2\ell$.  
Moreover: the map \eqref{hol} restricts to a local biholomorphism from
$\mathcal{R}(M , \mathcal{T}^+)$ to $(\bC^{\times})^{2\ell}$. 
\end{theorem}
\begin{proof}  The proof that $\mathcal{R}(M , \mathcal{T}^+)$ is  smooth follows imediately if 
we prove that $g$ is of constant rank at its points. We will show that the complex dimension
of  $\mathrm{Ker}(dg)$ is $2l$ and relate it to the map \ref{hol} in order to prove the second part of the theorem.

The key point of the proof of Theorem \ref{theo:variety} is the following:

\begin{lemma}\label{lem:A-Omega}Let $z \in \mathcal{R}(M , \mathcal{T}^+)$.
\begin{itemize}
 \item For every $\xi\neq 0$ in $ 
A(z)$, we have $\Omega^*(\xi,\bar \xi)\neq 0$,
 \item $(\bC \otimes (\mathrm{Im} (p \circ F))  \cap \mathcal{A} (z) = \{ 0 \}.$
\end{itemize}
\end{lemma}
\begin{proof} Here $\bar \xi$ is the complex conjugate of $ \xi$. The second point is a direct consequence of the first one. Indeed let 
$$\xi \in (\bC \otimes (\mathrm{Im} (p \circ F))  \cap \mathcal{A} (z) .$$
It follows from the first point in Proposition \ref{L4} that $\Omega^* (\xi ,
\overline{\xi}) =
0$. If the first point holds, then it forces $\xi$ to be null. 

Now $\Omega^* (\xi , \overline{\xi})$ can be computed locally on each tetrahedron $T_{\mu}$: 
Since $\xi$ belongs to the subspace $\bC \otimes J^* \subset \bC \otimes (J^2)^*$, it is
determined by the coordinates $\xi_{ij}^{\mu} = \xi (e_{ij}^{\mu})$.
Now, with respect to the
symplectic form $\Omega$, the basis vector $e_{ij}^{\mu}$ is
orthogonal to all the basis vectors except $e_{ik}^{\mu}$ and $\Omega
(e_{ij}^{\mu} , e_{ik}^{\mu}) = 1$.  By duality we therefore have   
\begin{equation*}
\begin{split}
\Omega^* (\xi , \overline{\xi}) &  = \sum_{\mu=1}^{\nu} \sum_{i=1}^4 ( \xi_{ij}^{\mu} \overline{\xi}_{ik}^{\mu} -\overline{\xi}_{ij}^{\mu} \xi_{ik}^{\mu}) \\
& =- \sum_{\mu=1}^{\nu} \sum_{i=1}^4 |\xi_{ij}^{\mu} |^2 \left( \frac{1}{\overline{z_{il} (T_{\mu})}} - \frac{1}{z_{il} (T_{\mu})}\right) .
\end{split}
\end{equation*} 
Here the last equality follows from the fact that $\xi \in \mathcal{A} (z)$. 
We conclude because for each $\mu$ and $i$ we have (up to a nonzero constant):
$$\mathrm{Im} \left( \frac{1}{\overline{z_{il} (T_{\mu})}} - \frac{1}{z_{il} (T_{\mu})}\right) >0.$$
\end{proof}

\subsection{} \label{lagrangian} Let $\mathcal L(z)$ be the image of $\mathcal A(z)$ in
$\oplus_{s=1}^{\ell} H^1 (T_s,\bC^2)$. It follows from the previous
lemma and the fact that the map is defined over $\Q$ (see Lemma
\ref{L4}) that $\mathcal L(z)$ is a totally isotropic subspace
isomorphic to $\mathcal A(z)\cap (\bC\otimes(J^*\cap \mathrm{Ker}(F^*))$ and
satisfies that for any $\chi\neq 0$ in $\mathcal L(z)$, we have
$\textrm{wp}(\chi,\bar\chi)\neq 0$. 

The space $\oplus_{s=1}^{\ell} H^1 (T_s ,\bC^2)$ decomposes as the sum
of two subspaces: $\sum_s [a_s]\otimes \bC^2$ and $\sum_s [b_s]\otimes
\bC^2$ (where $[a_s]$, resp $[b_s]$, denotes the Poincar\'e dual to
$a_s$, resp $b_s$). Both are Lagrangian subspaces and are invariant
under complex conjugation. To prove theorem \ref{theo:variety}, it
remains to prove that  
$\mathcal L (z)$ projects surjectivily onto $\sum_s [a_s]\otimes
\bC^2$. The dimension $\dim \mathcal L(z)$ may be computed. In fact,
by duality, we have: 
$$\dim (J^* \cap \mathrm{Ker}(F^*)) =  \dim (\mathrm{Im} ( p) \cap
\mathrm{Ker}(F^*)) =  \dim (J^2)^* -  \dim (\mathrm{Im} (F) + \mathrm
{Ker}( p)).$$ 
But we obviously have:
$$\dim (\mathrm{Im} (F) + \mathrm {Ker}( p)) = \dim \mathrm{Ker} ( p)
+ \dim \mathrm{Im} (F) - \dim (\mathrm {Ker}(p)\cap \mathrm{Im}(F)).$$ 
On the other hand we have $\dim J^2 = 16\nu$, $\dim \mathrm{Ker} ( p)
= 8\nu$ and\footnote{Note that the map $F$ is injective.} $\dim
\mathrm{Im} (F) = \dim C_1^{\rm or} + \dim C_2 = 4 \nu$. It finally
follows from the proof of  
\cite[Lemma 7.13]{BFG} that $\dim (\mathrm {Ker}(p)\cap \mathrm{Im}(F)) = 2\ell$. We conclude that 
$$\dim(J^*\cap \mathrm{Ker}(F^*))=4\nu+2 \ell .$$ 
Now $\dim \mathcal A(z)=4\nu$. The intersection $\mathcal A(z)\cap
J^*\cap \mathrm{Ker}(F^*)$ is therefore of dimension at least $2\ell$
and $\mathcal L(z)$ is a totally isotropic subspace of dimension at
least $2\ell$ in a symplectic space of dimension $4\ell$: it is a
Lagrangian subspace. Theorem \ref{theo:variety} now immediately
follows from the following lemma. 
\end{proof}

\begin{remark}
 The preceeding considerations give a combinatorial proof that the image of
$\mathcal R(M,\mathcal T)$ is a Lagragian subvariety of the space of
representations of the fundamental group of the boundary of $M$.
\end{remark}

\begin{lemma}
We have:
$$\mathcal L(z)\cap \sum_s [b_s]\otimes \bC^2=\{0\}.$$
 \end{lemma}
\begin{proof}
 Suppose that $\chi$ belongs to this intersection. Since $\sum_s
 [b_s]\otimes \bC^2$ is a Lagrangian subspace invariant under complex
 conjugation, the complex conjugate $\bar \chi $ also belongs to
 $\sum_s [b_s]\otimes \bC^2$ and we have  
$$\textrm{wp}(\chi,\bar\chi)= 0.$$
Since $\chi$ also belongs to $\mathcal L(z)$, Lemma \ref{lem:A-Omega}
finally implies that $\chi = 0$. 
\end{proof}

\subsection{Rigid points}\label{ss:tranverse} In general if $z \in
\mathcal{R} (M , \mathcal{T})$, the space $\mathcal L (z)$ is still a Lagrangian
subspace. Replacing Lemma \ref{lem:A-Omega} by the {\it assumption} that 
\begin{equation} \label{transverse}
(\bC \otimes (\mathrm{Im} (p \circ F))  \cap \mathcal{A} (z) = \{ 0 \},
\end{equation}
the proof of Theorem \ref{theo:variety} still implies that $\mathcal{R} (M ,
\mathcal{T})$ is (locally around $z$) a  smooth complex manifold of dimension
$2\ell$ and the choice of a $2\ell$-dimensional subspace of 
$\oplus_{s=1}^{\ell} H^1 (T_s,\bC^2)$ transverse to $\mathcal L (z)$ yields a
choice of local coordinates. A point $z$ verifying \eqref{transverse} is
called a \emph{rigid point} of $\mathcal{R} (M ,
\mathcal{T})$: indeed, at such a point, you cannot deform the representation
without deforming its trace on the boundary tori. Note that if there exists a
point $z \in \mathcal{R} (M , \mathcal{T})$ such that the condition
\eqref{transverse} is satisfied, then \eqref{transverse} is satisfied for
almost
every point in the same connected component: this transversality condition may
be expressed as the non-vanishing of a determinant of a matrix with entries
in $\bC(z)$. In the next section we provide
explicit examples of all the situations that can occur.

\subsection{Proof of Theorem \ref{T1}}

Theorem \ref{T1} does not immediately follow from Theorem
\ref{theo:variety} since $M$ may not admit an ideal
triangulation. Recall however that $M$ has a finite regular cover $M'$
that do admit an ideal triangulation (\cite{LST}). We may therefore 
apply Theorem \ref{T1} to $M'$ and the proof follows from the general  (certainly well known) lemma.

\begin{lemma}\label{lem:fincov}
Let $M'$ be a finite regular cover of $M$. Let $\rho$ and $\rho'$ be the
geometric representations for $M$ and $M'$.

Then one cannot deform $\rho$ without deforming $\rho'$.
\end{lemma}

\begin{proof}
 Let $\gamma_i$ be a finite set of loxodromic element generating $\pi_1(M)$.
Let $n$ be the index of $\pi_1(M')$ in $\pi_1(M)$. Then $\gamma_i^n$ is a
loxodromic element of $\pi_1(M')$.

Hence $\rho'(\gamma_i^n)=(\rho(\gamma_i))^n$ is a loxodromic elements in
$\PGL(3,\bC)$. The crucial though elementary remark is that its $n$-th square
roots form a finite set of $\PGL(3,\bC)$. So, once
 $\rho'$ is fixed, the determination of a representation $\rho$ such that
$\rho'=\rho_{|\pi_1(M')}$ requires a finite number of choices:
we should choose a $n$-th square root for each $\rho'(\gamma_i^n)$ among a
finite number of them. 
\end{proof}

\section{Examples}

In this section we describe exact solutions of the compatibility equations which
give all unipotent decorations of the triangulation with two tetrahedra of the
figure eight knot's sister manifold. This manifold has one cusp, so is
homotopic to a compact manifold whose boundary consits of one torus. In term of
theorem \ref{theo:variety}, we are looking to the fiber over $(1,1)$ of the map
$z\mapsto (A,A^*)$. We show that beside rigid decorations (i.e. isolated
points in the fiber) we obtain non-rigid ones. Namely four $1$-parameter families
of unipotent decorations.

Among the rigid decorations, one corresponds to the
(complete) hyperbolic structure and belongs to $\mathcal R(M,\mathcal T^+)$. The
rigidity then follows from theorem \ref{theo:variety}. At the other isolated
points, the rigidity is merely explained by the transversality between
$\mathcal A(z)$ and $\mathrm{Im}(p\circ F)$, as explained in subsection
\ref{ss:tranverse}.

As for the non-rigid components, their existence shows firstly that rigidity is
not granted at all. Moreover, the geometry of the
fiber over a point in $(\mathbb C^*)^2$ appears to be possibly complicated, with
intersections of components. The map from the
(decorated) representation variety $\mathcal R(M,\mathcal T)$ to its image in the representation
variety of
the torus turns out to be far from trivial from a geometric point of view.

Let us stress out that these components contain also points of special interest:
there are points corresponding to representations with value in
$\mathrm{\PSL}(2,\bC)$ which are rigid \emph{inside} $\mathrm{\PSL}(2,\bC)$, but not
anymore inside $\mathrm{\PSL}(3,\bC)$.

The analysis of this simple example seems to indicate that basically anything
can happen, at least outside of $\mathcal R(M,\mathcal T^+)$.

\subsection{The figure-eight knot's sister manifold}
This manifold $M$ and its triangulation $\mathcal T$ is described by the gluing
of two tetrahedra as in Figure \ref{sister}. Let $z_{ij}$ and $w_{ij}$ be the
coordinates associated to the edge $ij$. We will express all the equations in
terms of these edge coordinates (as the face coordinates are monomial in edges
coordinates, see \eqref{Rel1}).

\begin{figure}[ht]
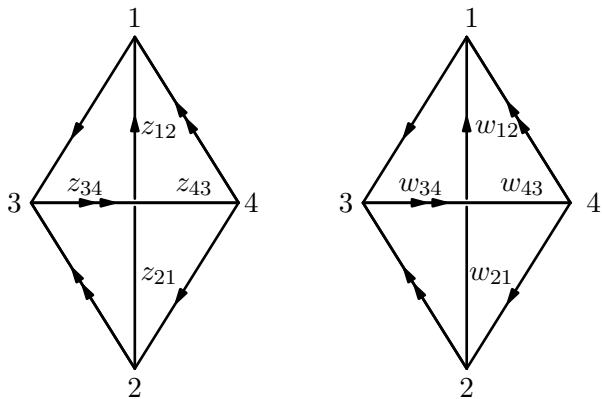
 
\begin{center}
\begin{asy}
  // 
  // 
  size(8cm);
  defaultpen(1);
  usepackage("amssymb");
  import geometry;
  
  // 
  point o = (8,0);
  point oo = (24,0);
  pair ph = (10,0);
  pair pv = (5,8);

  // 
  draw(o+pv--o, Arrow(5bp,position=.6));
  draw( o+ph--o+pv, Arrow(5bp,position=.5),Arrow(5bp,position=.6));
  draw(o+ph-pv-- o, Arrow(5bp,position=.5),Arrow(5bp,position=.6));
  draw(o+ph--o+ph-pv, Arrow(5bp,position=.6));
  draw(o+ ph/2+(0,.2)--o+pv, Arrow(5bp,position=.5));
  draw(o+ph-pv -- o+ph/2+ (0,-.2));
  draw(o -- o+ph,Arrow(5bp,position=.3), Arrow(5bp,position=.4) );
  
  // 
  label("\small { $z_{34}$}", o, 3*dir(25));
  //
  //
  label("{\small  $3$}", o, 1*W);
  
  label("\small {$z_{43}$}", o+ph, 3*dir(180-25));
  //
  //
  label("{\small $4$}", o+ph, .6*E);
  
  //
  label(" \small $ z_{12}$", o+pv, 9*dir(-80));
  //
  label("{\small $1$}", o+pv, 1*N);
  
   //
  label(" \small $ z_{21}$", o+ph-pv, 9*dir(80));
  //
  label("{\small $2$}", o+ph -pv, 1*S);

  // 
  draw( oo+pv -- oo,Arrow(5bp,position=.6));
  draw( oo+ph -- oo+pv ,Arrow(5bp,position=.5),Arrow(5bp,position=.6));
  draw( oo+ph-pv -- oo,Arrow(5bp,position=.5),Arrow(5bp,position=.6));
  draw(oo+ ph/2+(0,.2)--oo +pv, Arrow(5bp,position=.5));
  draw(oo+ph-pv -- oo+ ph/2-(0,.2));
  draw(oo+ ph--oo+ph-pv ,Arrow(5bp,position=.6));
  draw(oo -- oo+ph,Arrow(5bp,position=.3), Arrow(5bp,position=.4));
  
   // 
  label("\small { $w_{34}$}", oo, 3*dir(25));
  //
  //
  label("{ \small $3$}", oo, 1*W);
  
  label("\small {$w_{43}$}", oo+ph, 3*dir(180-25));
  //
  //
  label("{ \small $4$}", oo+ph, .6*E);
  
  //
  label(" \small $  w_{12}$", oo+pv, 9*dir(-77));
  //
  label("{ \small $1$}", oo+pv, 1*N);
  
   //
  label(" \small $ w_{21}$", oo+ph-pv, 9*dir(80));
  //
  label("{\small $2$}", oo+ph -pv, 1*S);

\end{asy}
\caption{The figure eight sister manifold represented by two tetrahedra.} \label{sister}
\end{center}
\end{figure}

The variety $\mathcal R(M,\mathcal T)$ is then given by relations \eqref{Rel2}
and \eqref{Rel3} among the $z_{ij}$ and among the $w_{ij}$ plus the face and
edge conditions \eqref{face} and \eqref{edge}.

In this case, the edge equations are:
\begin{equation}
(L_e)\left \{
\begin{array}{rcl}
e_1:=z_{23}z_{34}z_{41}w_{23}w_{34}w_{41}-1&=&0,\\
e_2:=z_{32}z_{43}z_{14}w_{32}w_{43}w_{14}-1&=&0,\\
e_3:=z_{12}z_{24}z_{31}w_{12}w_{24}w_{31}-1&=&0,\\
e_4:=z_{21}z_{42}z_{13}w_{21}w_{42}w_{13}-1&=&0.
\end{array}
\right .
\end{equation}

and the face equations are:

\begin{equation}
(L_f)\left \{
\begin{array}{rcl}
f_1:=z_{21}z_{31}z_{41}w_{12}w_{32}w_{42}-1&=&0,\\
f_2:=z_{12}z_{32}z_{42}w_{21}w_{31}w_{41}-1&=&0,\\
f_3:=z_{13}z_{43}z_{23}w_{14}w_{34}w_{24}-1&=&0,\\
f_4:=z_{14}z_{24}z_{34}w_{13}w_{23}w_{43}-1&=&0.
\end{array}
\right .
\end{equation}

Moreover, one may compute the eigenvalues of the holonomy in the boundary
torus (see, \cite[section 7.3]{BFG}) by following the two paths representing the generators of 
the boundary torus homology in Figure \ref{holonomy}. The two eigenvalues associated to a path
are obtained using the following rule: For the first one say $A$, we multiply the cross-ratio invariant $z_{ij}$ if the vertex $ij$ of a triangle is seen to the left
and by its inverse if its seen to the right.  For the inverse of the second one, say $A^*$, we multiply by  $1/z_{ji}$ if the vertex $ij$ of a triangle is seen to the left
and by $z_{kl}z_{lk}/z_{ij}$ if it is seen to the right.

$$
A= z_{12}\frac{1}{w_{32}}
z_{41}\frac{1}{w_{21}},A^*=
\frac{1}{z_{21}}\frac{w_{14}w_{41}}{w_{32}}\frac{1}{z_{14}}\frac{w_{34}w_{43}}{w_{21}},
$$
$$
B=
z_{31}\frac{1}{w_{14}}z_{42}\frac{1}{w_{23}},B^*=
\frac{1}{z_{13}}\frac{w_{23}w_{32}}{w_{14}}\frac{1}{z_{24}}\frac{w_{14}w_{41}}{w_{23}}
$$
or, equivalently~:

\begin{equation}
(L_{h,A,A^{*},B,B^{*}})\left \{
\begin{array}{lclcr}
h_{A}&:=&w_{32}w_{21}A- z_{12} z_{41}&=&0,\\ 
h_{A^{*}}&:=&z_{21}w_{32}z_{14}w_{21}A^{*}-w_{14}w_{41}w_{34}w_{43}&=&0,\\
h_{B}&:=&w_{14}w_{23}B-z_{31}z_{42}&=&0,\\
h_{B^{*}}&:=&z_{13}w_{14}z_{24}w_{23}B^{*}-w_{23}w_{32}w_{14}w_{41}&=&0.
\end{array}
\right .
\end{equation}

\begin{figure}
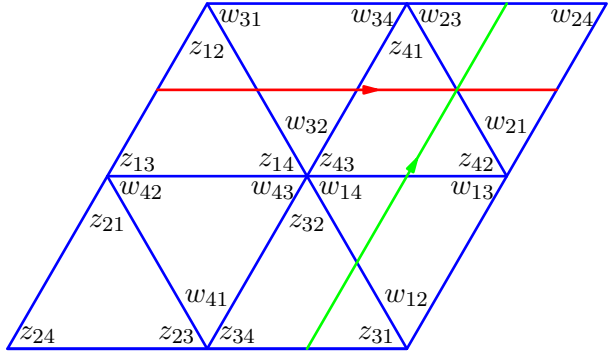
\label{holonomy}
\begin{center}
\begin{asy}
  // 
  // 
  size(8cm);
  defaultpen(1);
  usepackage("amssymb");
  import geometry;
  
  // 
  point o = (1,0);
  pair ph = (2,0);
  pair pv = (1,1.73);
  
  // 
  draw(o -- o+2*pv -- o+2*pv+2*ph--o+2*ph--cycle,blue);
  draw( o+2*pv+ph -- o+pv+2*ph,blue);
  draw( o+pv -- o+pv+ph --cycle,blue);
  draw( o+pv -- o+ph --cycle,blue);
  draw( o+2*pv -- o+2*ph --cycle,blue);
  draw( o+pv+ph --o+pv+2*ph--cycle,blue);
  draw(o -- o+2*ph, blue);
  draw(o +ph-- o+2*pv+ph, blue);
  
  // 
  label("{\small $z_{24}$}", o, 1.5*dir(35));
  label("\small $ z_{34}$", o+ph, 1.5*dir(35));
  label("\small $z_{43}$", o+ph+pv, 1.5*dir(35));
  label("\small $z_{13}$", o+pv, 1.5*dir(35));
  
  label("\small $ z_{21}$", o+pv, 4*S);
  label("\small $ z_{32}$", o+pv+ph, 4*S);
  label("\small $z_{41}$", o+2*pv+ph, 4*S);
  label("\small $ z_{12}$", o+2*pv, 4*S);
  
  label("\small $z_{23}$", o+ph, 1.5*dir(145));
  label("\small $z_{31}$", o+2*ph, 1.5*dir(145));
  label("\small $z_{42}$", o+pv+2*ph, 1.5*dir(145));
  label("\small $ z_{14}$", o+ph+pv, 1.5*dir(145));

  label("\small$w_{42}$", o+pv, 1.5*dir(-35));
  label("\small $w_{14}$", o+pv+ph, 1.5*dir(-35));
  label("\small $ w_{31}$", o+2*pv, 1.5*dir(-35));
  label("\small $w_{23}$", o+2*pv+ph, 1.5*dir(-35));
  
  label("\small$ w_{41}$", o+ph, 4.5*N);
  label("\small $ w_{12}$", o+ph+ph, 4.5*N);
  label("\small $w_{32}$", o+pv+ph, 4.5*N);
  label("\small $ w_{21}$", o+2*ph+pv, 4.5*N);

  label("\small $ w_{43}$", o+pv+ph, 1.5*dir(-145));
  label("\small $ w_{13}$", o+pv+2*ph, 1.5*dir(-145));
  label("\small $w_{34}$", o+2*pv+ph, 1.5*dir(-145));
  label("\small $ w_{24}$", o+2*pv+2*ph, 1.5*dir(-145));

 // 
 // 
 draw ((o+1.5*pv)..((o+1.5*pv)+2*ph),red,Arrow(5bp,position=.55));
 
// 
draw ((o+1.5*ph)..((o+1.5*ph)+2*pv),green,Arrow(5bp,position=.55));

\end{asy}
\caption{The boundary holonomy of the figure eight sister manifold. The red line corresponds to $A, A^*$ and the green line to $B, B^*$} 
\end{center}
\end{figure}

If $A=B=A^*=B^*=1$ the solutions of the equations correspond to unipotent 
structures.  
\def\Frac#1#2{{\displaystyle{{#1} \overwithdelims.. {#2}}}}
\subsection{Methods}
The computational problem to be solved is the description of a
constructible set of ${\bC}^{24}$ defined by the union of the edge
equations ($L_e$), the face equations ($L_f$), the equations
modelizing unipotent structures ($L_{h,1,1,1,1}$) augmented by a set
of relations between some of the variables ($L_r$) and a set of
inequalities (the coordinates are supposed to be
different from $0$ and $1$), with~:

\begin{equation}
L_r:=\left \{
\begin{array}{cccc}
w_{13}= \Frac 1{1-w_{12}},&
w_{14}={\Frac {w_{12}-1}{w_{12}}},&
w_{23}={\Frac {w_{21}-1}{w_{21}}},&
w_{24}= \Frac 1{1-w_{21}},\\
w_{31}= \Frac 1{1-w_{34}},&
w_{32}={\Frac {w_{34}-1}{w_{34}}},&
w_{41}={\Frac {w_{43}-1}{w_{43}}},&
w_{42}= \Frac 1{1-w_{43}},\\
z_{13}= \Frac 1{1-z_{12}},&
z_{14}={\Frac {z_{12}-1}{z_{12}}},&
z_{23}={\Frac {z_{21}-1}{z_{21}}},&
z_{24}= \Frac 1{1-z_{21}},\\
z_{31}= \Frac 1{1-z_{34}},&
z_{32}={\Frac {z_{34}-1}{z_{34}}},&
z_{41}={\Frac {z_{43}-1}{z_{43}}},&
z_{42}= \Frac 1{1-z_{43}}.
\end{array}
\right .
\end{equation}

After a straightforward substitution of the relations $L_r$ in the
equations
$$\{e_1,\ldots,e_4,f_1,\ldots,f_4,{h_{A}}_{|A=1},{h_{A^{*}}}_{|A^*=1},{h_{B}}_{
|B=1 } , { h_
{ B^
{*}}}_{|B^*=1}\},$$
one shows that the initial problem is then equivalent to describing the
constructible set defined by a set of $12$ polynomial equations 
$${\cal E}:=\left \{x\in {\bC}^8,P_i(x)=0, \ i=1, \ldots ,12, \ P_i\in\Z[\cal{X}] \right \},$$
in $8$ unknowns 
$${\cal X}=\{z_{12},\,
z_{21},\,z_{34},\,z_{43},\,w_{12},\,w_{21},\,w_{34},\,w_{43}\},$$
and a set  of $16$
polynomial inequalities ${\cal F}:= \left \{x\in {\bC}^{8},u(x)\neq 0,u(x)\neq 1,
u\in {\cal
    X}\right \}$.
Classical tools from computer algebra are used to:
\begin{itemize}
  \item Compute generators of ideals using Gr\"obner bases.  A Gr\"obner basis of a polynomial ideal $I$ is a set of generators of $I$, such that there is a
natural way of reducing canonically a polynomial $P$ ($\rm mod$ $I$). 
  \item Eliminate variables: Given $\mathcal{Y \subset \mathcal{X}}$ and
   $I\subset \mathbbm{Q} \left [ \mathcal{X} \right ]$, compute $J = I \cap
  \mathbbm{Q} \left[ \mathcal{Y} \right]$ and note that the set $\mathcal{J} =
  \left\{ x \in \mathbbm{C}^{\sharp \mathcal{Y}}, p \left( x \right) = 0, p
  \in J \right\}$ is the Zariski closure of the projection of $\mathcal{I} =
  \left\{ x \in \mathbbm{C}^{\sharp \mathcal{X}}, p \left( x \right) = 0, p
  \in I \right\}$ onto the $\mathcal{Y}$-coordinates.
\end{itemize}
Combining the items, one can then compute an ideal $I'$ whose zero set is
$\overline{\mathcal{E} \setminus \mathcal{F}}$ by computing $\left( I +
\langle T \prod_{f \in \mathcal{F}} f - 1 \rangle \right) \cap \mathbbm{Q}
\left[ \mathcal{X} \right]$ (see for example {\cite[chapter 4]{CLOS}}).

For rather small systems, one then compute straightforwardly (by means of a
classical algorithm) a prime or primary decomposition of any ideal defining
$\overline{\mathcal{E} \setminus \mathcal{F}}$. This is possible in the
present case. In practice however, for triangulations with more than two
tetrahedra, these classical algorithms will not be sufficiently powerful to
study these varieties.

We do not go further in the description of the computations which will be part
of a more general contribution by the last three authors. Let us just mention
that the process gives us an exhaustive description of all the components of
the constructible set we study. Moreover, the interested reader may easily
check that the given solutions verify indeed all the equations.

For the present paper, we just retain that a prime decomposition of an ideal
defining $\overline{\mathcal{E} \setminus \mathcal{F}}$ has been computed and
we give the main elements describing the solutions so that the reader can at
least check the main properties (essentially dimensions) of the results.

Each component ($0$ or $1$ dimensional) can be described in the same way: a
polynomial $P$ (in one or $2$ variables) over $\mathbbm{Q}$ such that each
coordinate $z_{ij}$ or $w_{ij}$ is an algebraic (over $\mathbbm{Q}$) function
of the roots of $P$. In particular, they naturally come in families of Galois
conjugates. This is no surprise, as the equations defining $\mathcal{R} (M,
\mathcal{T})$ have integer coefficients.




\subsection{Rigid unipotent decorations}

We are looking for the isolated points of the set $\mathcal U=\{z\in\mathcal
R(M,\mathcal T) \; | \; A=A^*=B=B^*=1\}$.

There are 4 Galois families of such points.
They are described by four irreducible polynomials with integer coefficients in
one variable.  Two of them are of degree 2 and  the other two of degree 8.

The first polynomial is the minimal polynomial of
the sixth root of unity $\frac{1+i\sqrt{3}}{2}$. For a root
$\omega^\pm=\frac{1 \pm i\sqrt{3}}{2}$, the following defines an isolated point
in $\mathcal U$:
$$
z_{12}=z_{21}=z_{34}=z_{43}=w_{12}=w_{21}=w_{34}=w_{43}=\omega^\pm
$$ 
The solution associated to $\omega^+$ is easily checked to correspond to the
hyperbolic structure on $M$: it is the geometric representation as we called
it. The other one is its complex conjugate.

A point of $\mathcal R(M,\mathcal T)$ corresponding to a representation in $\mathrm{PU}(2,1)$ (we call such representations CR, see \cite{falbeleight})
 with unipotent boundary holonomy was obtained in 
\cite{genzmer} and is parametrized by the same polynomial, the $z$ and $w$
coordinates being this time given by:
$$
z_{12}= z_{21}=-\omega \ \ \ z_{34}=z_{43}=-(\omega^\pm)^2,
$$
$$
w_{12}= w_{21}=-\omega^2 \ \ \ w_{34}=w_{43}=-\omega^\pm.
$$

The two other isolated 0-dimensional components have degree 8 and
their minimal polynomial are respectively:
$$
P={X}^{8}-{X}^{7}+5\,{X}^{6}-7\,{X}^{5}+
7\,{X}^{4}-8\,{X}^{3}+5\,{X}^{2}-2\,X+
1 = 0
$$
and
$$
Q(X) = P(1-X)=0. 
$$
We do not describe all the $z$ and $w$ coordinates in terms of their roots (for
the record, let us mention that $z_{43}$ is directly given by the root). None of
these 16 representations are in $\mathrm {PSL}(2,\mathbb
C)$ nor in $\mathrm{PU}(2,1)$. 

Although the computations above are exact we could also check that these isolated components are rigid by computing that the tangent space is zero dimensional.
We do not include the computations here. 

\subsection{Non-rigid components}

There exist two 1-dimensional prime components  ($S_1$ and $S_2$) each of them can be parametrized 
by two 1-parameter families.

The four 1-parameter families of solutions are described as follows:
let $\tau^{\pm} = \frac 12 \pm \frac 12 \sqrt 5$ be one of the two real
roots of $X^2 = X+1$. 
Then the roots $X^2 - XY -Y^2$ define two 1-parameter families 
meeting at $(0,0)$: $X=\tau^{\pm} Y$.
They parametrize four 1-parameter families of points
$(S_1^{\pm})$ and $( S_2^{\pm})$.

For $S_1$ we obtain: 
$$
(S_1^{\pm}) \quad \left \{
\begin{array}{ll}
z_{12} = w_{12} = 
\Frac {X+Y}{X-1}, &
z_{21}=w_{21} = 1+Y\\
z_{34}=w_{34} = 
\Frac {X^2+X+Y}{X(X-1)}, & 
z_{43}=w_{43}=X.
\end{array}
\right .
$$
 By restricting  $S_1$ to the conditions so that the representation be in $\PU(2,1)$ we obtain (after writing the system as a real system separating real and imaginary parts) an algebraic set of real dimension 1 entirely characterized by its projection 
on the coordinates in $\bR^2$ of $z_{21}=x+i y$.  The projection is a product of two circles:
$$
(x-\tau^\pm)^2 + y^2 = 1. 
$$

Among the solutions (in $S_1$) we obtain only two  belonging to $\mathrm{PSL}(2,\bC)$ (and
they even belong to
$\mathrm{PSL}(2,\R)\subset \PU(2,1)$):
\begin{eqnarray*}
z_{12}=z_{21}=z_{34}=z_{43}=w_{12}=w_{21}=w_{34}=w_{43}=1+\tau^\pm.\\
\end{eqnarray*}
These points are then rigid inside $\mathrm{PSL}(2,\bC)$ but not inside
$\mathrm{PSL}(3,\bC)$ (neither inside $\PU(2,1)$).

The other two 1-parameter families are parametrized as follows
$$
(S_2^{\pm}) \quad 
\left \{
\begin{array}{ll}
z_{12} = w_{21} = 
1+\Frac {Y}{X} - \Frac {(X+1)(Y+1)}{X^2+X-1} + ,&
z_{21} = w_{12} = 
\Frac {X+Y-1}{Y-1},\\
z_{34}=w_{43} = X+Y, &
z_{43} = w_{34} = 1/Y
\end{array}
\right .
$$
None of these points gives a representation in  $\PSL(2,\bC)$ nor in $\PU(2,1)$.

\bibliography{bibli}

\bibliographystyle{plain}

\end{document}

%% file: Wcombinatorics.pdf_tex
\begingroup%
  \makeatletter%
  \providecommand\color[2][]{%
    \errmessage{(Inkscape) Color is used for the text in Inkscape, but the package 'color.sty' is not loaded}%
    \renewcommand\color[2][]{}%
  }%
  \providecommand\transparent[1]{%
    \errmessage{(Inkscape) Transparency is used (non-zero) for the text in Inkscape, but the package 'transparent.sty' is not loaded}%
    \renewcommand\transparent[1]{}%
  }%
  \providecommand\rotatebox[2]{#2}%
  \ifx\svgwidth\undefined%
    \setlength{\unitlength}{177.90959473bp}%
    \ifx\svgscale\undefined%
      \relax%
    \else%
      \setlength{\unitlength}{\unitlength * \real{\svgscale}}%
    \fi%
  \else%
    \setlength{\unitlength}{\svgwidth}%
  \fi%
  \global\let\svgwidth\undefined%
  \global\let\svgscale\undefined%
  \makeatother%
  \begin{picture}(1,0.65284719)%
    \put(0,0){\includegraphics[width=\unitlength]{Wcombinatorics.pdf}}%
    \put(0.46767081,0.61185024){\color[rgb]{0,0,0}\makebox(0,0)[lb]{\smash{$i$}}}%
    \put(0.90448976,0.0112241){\color[rgb]{0,0,0}\makebox(0,0)[lb]{\smash{$k$}}}%
    \put(-0.0044791,0.0112241){\color[rgb]{0,0,0}\makebox(0,0)[lb]{\smash{$j$}}}%
  \end{picture}%
\endgroup%

%% file: coordinates.pdf_tex
\begingroup%
  \makeatletter%
  \providecommand\color[2][]{%
    \errmessage{(Inkscape) Color is used for the text in Inkscape, but the package 'color.sty' is not loaded}%
    \renewcommand\color[2][]{}%
  }%
  \providecommand\transparent[1]{%
    \errmessage{(Inkscape) Transparency is used (non-zero) for the text in Inkscape, but the package 'transparent.sty' is not loaded}%
    \renewcommand\transparent[1]{}%
  }%
  \providecommand\rotatebox[2]{#2}%
  \ifx\svgwidth\undefined%
    \setlength{\unitlength}{222.04707031bp}%
    \ifx\svgscale\undefined%
      \relax%
    \else%
      \setlength{\unitlength}{\unitlength * \real{\svgscale}}%
    \fi%
  \else%
    \setlength{\unitlength}{\svgwidth}%
  \fi%
  \global\let\svgwidth\undefined%
  \global\let\svgscale\undefined%
  \makeatother%
  \begin{picture}(1,0.81858688)%
    \put(0,0){\includegraphics[width=\unitlength]{coordinates.pdf}}%
    \put(0.54895741,0.78573912){\color[rgb]{0,0,0}\makebox(0,0)[lb]{\smash{$i$}}}%
    \put(-0.00358877,0.24099334){\color[rgb]{0,0,0}\makebox(0,0)[lb]{\smash{$j$}}}%
    \put(0.92347484,0.25205013){\color[rgb]{0,0,0}\makebox(0,0)[lb]{\smash{$k$}}}%
    \put(0.52538854,0.00635423){\color[rgb]{0,0,0}\makebox(0,0)[lb]{\smash{$l$}}}%
    \put(0.34515013,0.68201553){\color[rgb]{0,0,0}\makebox(0,0)[lb]{\smash{$z_{ij}$}}}%
    \put(0.55077708,0.56737403){\color[rgb]{0,0,0}\makebox(0,0)[lb]{\smash{$z_{il}$}}}%
    \put(0.66359914,0.65835936){\color[rgb]{0,0,0}\makebox(0,0)[lb]{\smash{$z_{ik}$}}}%
    \put(0.49254638,0.41269876){\color[rgb]{0,0,0}\makebox(0,0)[lb]{\smash{$z_{ijk}$}}}%
    \put(0.3251333,0.33263156){\color[rgb]{0,0,0}\makebox(0,0)[lb]{\smash{$z_{ilj}$}}}%
    \put(0.64482958,0.33950299){\color[rgb]{0,0,0}\makebox(0,0)[lb]{\smash{$z_{ikl}$}}}%
  \end{picture}%
\endgroup%